\documentclass[11pt]{article}

\usepackage{ amsthm, amsmath,amssymb}

\def \[{\begin{equation}}
\def \]{\end{equation}}

\def\bdes{\begin{description}}
\def\edes{\end{description}}
\def\benu{\begin{enumerate}}
\def\eenu{\end{enumerate}}
\def\bitm{\begin{itemize}}
\def\eitm{\end{itemize}}

\def\R{{\sl I\kern-3.2pt R}}

\def\sqr#1#2{{\vcenter{\hrule height .#2pt
      \hbox{\vrule width .#2pt height#1pt \kern#1pt\vrule width.#2pt}
                       \hrule height.#2pt}}}

\newtheorem{lemma}{Lemma}[section]

\newtheorem{thm}{Theorem}[section]

\begin{document}

\title{{ \Large \bf Symmetry and Nonexistence of Positive Solutions for Fractional
Choquard Equations}
}
\author{\;\;Pei Ma, Jihui Zhang\thanks{Corresponding author.}}

\date{\today}
\maketitle

\begin{abstract}
This paper is devoted to study the following Choquard equation
\begin{eqnarray*}\left\{
\begin{array}{lll}
 (-\triangle)^{\alpha/2}u=(|x|^{\beta-n}\ast u^p)u^{p-1},~~~&x\in R^n,\\
 u\geq0,\,\,&x\in R^n,
 \end{array}
 \right.
\end{eqnarray*}
where $0<\alpha,\beta<2$, $1\leq p<\infty$, and $n\geq2$.
Using a direct method of moving planes, we prove the symmetry and nonexistence of positive solutions in the critical and subcritical case respectively.
\end{abstract}


\section{Introduction}
We study the following Choquard equation involving the fractional Laplacian
\begin{eqnarray}\label{1.1}
\left\{
\begin{array}{ll}
 (-\triangle)^{\alpha/2}u=(|x|^{\beta-n}\ast u^p)u^{p-1},\,\,&x\in R^n,\\
 u\geq0,\,\,&x\in R^n,
 \end{array}
 \right.
\end{eqnarray}
where $0<\alpha, \beta<2$, $1\leq p<\infty$ and $n\geq2$.

 The fractional Laplacian in $R^n$ is a nonlocal pseudo-differential operator taking the form
\begin{eqnarray}\label{1.2}
(-\triangle)^{\alpha/2}u(x)=C_{n,\alpha}PV
\int_{R^n}\frac{u(x)-u(y)}{|x-y|^{n+\alpha}}dy=C_{n,\alpha}\lim_{\varepsilon\rightarrow
0} \int_{R^n\setminus
B_\varepsilon(x)}\frac{u(x)-u(y)}{|x-y|^{n+\alpha}}dy,
\end{eqnarray}
where $C_{n,\alpha}$ is a normalization constant. This operator is well defined in
$\mathcal{S}$, the Schwartz space of rapidly decreasing $C^\infty$
functions in $R^n$. In this space, it can also be equivalently
defined in terms of the Fourier transform
\begin{eqnarray*}
\mathcal{F}[(-\triangle)^{\alpha/
2}u](\xi)=|\xi|^\alpha\mathcal{F}u(\xi),
\end{eqnarray*}
where $\mathcal{F}u$ is the Fourier transform of $u$. One can extend
this operator to a wider space of distributions:
\begin{eqnarray*}
\mathcal{L}_\alpha= \{u:R^n\rightarrow
R\mid\int_{R^n}\frac{|u(x)|}{1+|x|^{n+\alpha}}dx<\infty\}.
\end{eqnarray*}
Then in this space, we defined $(-\triangle)^{\alpha/2}u$ as a
distribution by

\begin{eqnarray*}
\langle(-\triangle)^{\alpha/2}u(x), \phi\rangle=\int_{R^n}
u(x)(-\triangle)^{\alpha/2}\phi(x)dx, \,\,\forall\phi\in C_0^\infty(R^n).
\end{eqnarray*}

In our paper, let
\begin{equation}\label{011801}
v(x)=|x|^{\beta-n}\ast u^p=\int_{R^n}\frac{u^p(y)}{|x-y|^{n-\beta}}dy,
\end{equation}
and act $(-\triangle)^{\beta/2}$ on both side of (\ref{011801}), we obtain
$$(-\triangle)^{\beta/2}v(x)=u^p(x).$$
Then, (\ref{1.1}) is equivalent to
\begin{eqnarray}\label{011802}
\left\{
\begin{array}{ll}
(-\triangle)^{\alpha/2}u=v(x)u^{p-1}(x),\,\,&x\in R^n,\\
(-\triangle)^{\beta/2}v=u^{p}(x),\,\,&x\in R^n,\\
u\geq0,v\geq 0,\,\,&x\in R^n.
\end{array}
\right.
\end{eqnarray}
Hence, to study (\ref{1.1}), it is sufficiently to investigate (\ref{011802}).

In recent years, the fractional Laplacian has attracted much attention. It appears in diverse physical phenomena, such as anomalous diffusion and quasi-geostrophic flows. It also has various applications in probability and finance. In particular, the fractional Laplacian can be understood as the infinitesimal generator of a stable L\'{e}vy diffusion process and appear in anomalous diffusions in plasmas, flames propagation and chemical reactions in liquids, population dynamics, geographical fluid dynamics, and American options in finance. For readers who are interested in the application of the fractional Laplacian, please refer to \cite{A}, \cite{B} and the references therein. 

In \cite{BCDS}, the authors considered the following fractional Laplacian equation
\begin{eqnarray}\label{021401}
	(-\triangle)^{\alpha/2}u=u^p,\,\, x\in R^n,
\end{eqnarray}
they used the extension method to deduce the nonlocal problem into a local one in a higher dimensional half space $R^n\times [0, \infty)$, then applied the method of moving planes to show the symmetry of $U(x,y)$ in $x$, then derived the nonexistence of positive solutions in the subcritical case. 

In \cite{CLL}, the authors developed a direct method of moving planes for the fractional Laplacian, and using this method, they derived the symmetry and nonexistence of positive solutions  for
\begin{eqnarray*}
\left\{
\begin{array}{ll}
&(-\triangle)^{\alpha/2}u=u^p,\\
&u\geq0,
\end{array}
\right.
\end{eqnarray*}
 in $R^n$ and $R^n_+$.

In \cite{LM}, the authors studied the system involving the fractional Laplacian
\begin{eqnarray*}
	\left\{
	\begin{array}{ll}
		(-\triangle)^{\alpha/2}u=f(v(x)),\,\,&x\in R^n,\\
		(-\triangle)^{\beta/2}v(x)=g(u(x)),\,\,&x\in R^n,\\
		u\geq0, v\geq0,\,\,&x\in R^n.
	\end{array}
	\right.
\end{eqnarray*}   
First, they used the iteration method to establish the maximum principles for system, then derived the symmetry of non-negative solutions by the direct method of moving planes without any decay assumption at infinity.

In our paper, we first establish the maximum principles by the iteration method introduced by \cite{LM}, and then used the direct method of moving planes
introduced by \cite{CLL} to derive the symmetry of positive solutions
and then deduce the nonexistence of positive solutions.

The following is our main theorems.

\begin{thm}(\textbf{Decay at Infinity})\label{t1}
 Let $\Omega$ be an unbounded region in $\Sigma_\lambda$. Assume $\varphi\in L_\alpha\cap C^{1,1}_{loc}(\Omega)$, $\phi\in L_\beta\cap C^{1,1}_{loc}(\Omega)$ and $\varphi(x), \phi(x)$ are lower semi-continuous. If
\begin{eqnarray}\label{1.3}
\left\{
\begin{array}{lll}
(-\triangle)^{\alpha/2}\varphi(x)+C_2(x)\varphi(x)+C_3(x)\phi(x)\geq 0 ~~~ ~~~&in& \Omega,\\
(-\triangle)^{\beta/2}\phi(x)+C_1(x)\varphi(x)\geq 0 ~~~ ~~~&\mbox{in}& \Omega,\\
\varphi(x),\phi(x)\geq0 ~~~ ~~~&\mbox{in}& \Sigma_\lambda\backslash\Omega,\\
\varphi(x^\lambda)=-\varphi(x)   ~~~ ~~~&\mbox{in}& \Sigma_\lambda,\\
\phi(x^\lambda)=-\phi(x)   ~~~ ~~~&\mbox{in}& \Sigma_\lambda,
\end{array}
\right.
\end{eqnarray}
with
\begin{eqnarray}\label{1.4}
C_1(x),C_3(x)\sim\frac{1}{|x|^{\alpha+\beta}}, \,\,C_2(x)\sim\frac{1}{|x|^{2\alpha}},\,\,\mbox{for} \,\,|x|\,\,\,\mbox{large},
\end{eqnarray}
and 
$$C_1(x), C_2(x), C_3(x)<0.$$
Then there exists a constant $R_0 > 0$ such that if
\begin{eqnarray}
\varphi(x_0) = \min_{\Omega}\varphi(x)<0,\,\,\phi(x_1)=\min_{\Omega}\phi(x)<0,
\end{eqnarray}
then at least one of $x_0$ and $x_1$ satisfies
\begin{eqnarray}\label{1.12}
|x|\leq R_0.
\end{eqnarray}
\end{thm}

\begin{thm}(\textbf{Narrow Region Principle})\label{t2}
 Let $\Omega$ be a bounded narrow region in $\Sigma_\lambda$, such that it is contained in $\{x|\lambda-\delta<x_1<\lambda\}$ with small $l$. Suppose that $\varphi(x)\in L_\alpha\cap C^{1,1}_{loc}(\Omega)$, $\phi(x)\in L_\beta\cap C^{1,1}_{loc}(\Omega)$ and $\varphi(x), \phi(x)$ are lower semi-continuous. If $C_1(x)$, $C_2(x)$ and $C_3(x)$ are bounded from below in $\Omega$, then
\begin{eqnarray}\label{1.5}
\left\{
\begin{array}{lll}
(-\triangle)^{\alpha/2}\varphi(x)+C_2(x)\varphi(x)+C_3(x)\phi(x)\geq 0 ~~~ ~~~&in& \Omega,\\
(-\triangle)^{\beta/2}\phi(x)+C_1(x)\varphi(x)\geq 0 ~~~ ~~~&in& \Omega,\\
\varphi(x),\phi(x)\geq0 ~~~ ~~~&\mbox{in}& \Sigma_\lambda\backslash\Omega,\\
\varphi(x^\lambda)=-\varphi(x)   ~~~ ~~~&\mbox{in}& \Sigma_\lambda,\\
\phi(x^\lambda)=-\phi(x)   ~~~ ~~~&\mbox{in}& \Sigma_\lambda,
\end{array}
\right.
\end{eqnarray}
then for sufficiently small $\delta$, we have
\begin{eqnarray}\label{2017010901}
\varphi(x),\phi(x) \geq 0,\,x\in\Omega.
\end{eqnarray}
Furthermore, if $\varphi = 0$ or $\phi(x)=0$ at some point in $\Omega$, then
\begin{eqnarray}\label{2017010902}
\varphi(x)=\phi(x) \equiv 0~~~\text{almost everywhere in}~~~R^n.
\end{eqnarray}
These conclusions hold for unbounded region $\Omega$ if we further
assume that
\begin{eqnarray}\label{2017010903}
\underline{\lim\limits}_{|x|\rightarrow\infty}\varphi(x),\phi(x)\geq0.
\end{eqnarray}
\end{thm}

\begin{thm}
	Let $0<\alpha,\beta<2$, $\frac{n}{n-\alpha}\leq p\leq\frac{n+\beta}{n-\alpha}$. Assume $u\in L_\alpha\cap C^{1,1}_{loc}$ and $v\in L_\beta\cap C^{1,1}_{loc}$ satisfy (\ref{1.1}). Then, 
	  
	(i) in the subcritical case $\frac{n}{n-\alpha}\leq p<\frac{n+\beta}{n-\alpha}$, (\ref{1.1}) has no positive solution;
	
   (ii) in the critical case $p=\frac{n+\beta}{n-\alpha}$, the positive solutions must be radially symmetric and monotone decreasing about some point in $R$.
	
\end{thm}
\section{Proof of Theorem \ref{t1} and \ref{t2}}
We first give some basic notations. Let
\begin{eqnarray*}
	T_\lambda=\{x\in R^n\mid x_1=\lambda,\lambda\in R\}
\end{eqnarray*}
be the moving plane,
\begin{eqnarray*}
	\Sigma_\lambda=\{x\in R^n\mid x_1<\lambda\}
\end{eqnarray*}
be the region to the left of the plane, $\tilde{\Sigma}_{\lambda}=R^n\backslash\Sigma_{\lambda}$, and
\begin{eqnarray*}
	x^\lambda=(2\lambda-x_1,x_2,...,x_n)
\end{eqnarray*}
be the reflection of the point $x = (x_1, x_2, \cdot\cdot\cdot ,
x_n)$ about the plane $\mathrm{T}_\lambda$.
\subsection{Decay at Infinity}
\textbf{Proof}. 
By the definition of the fractional Laplacian (\ref{1.2}),
\begin{eqnarray}\label{1.6}\nonumber
(-\triangle)^{\beta/2}\phi(x_1)&=&C_{n,\beta}PV\int_{R^n}\frac{\phi(x_1)-\phi(y)}{|x_1-y|^{n+\beta}}dy\\ \nonumber
&=&C_{n,\beta}PV\int_{\Sigma_\lambda}\frac{\phi(x_1)-\phi(y)}{|x_1-y|^{n+\beta}}dy+C_{n,\beta}PV\int_{\tilde{\Sigma}_{\lambda}}\frac{\phi(x_1)-\phi(y)}{|x_1-y|^{n+\beta}}dy\\ \nonumber
&=&C_{n,\beta}PV\int_{\Sigma_\lambda}\frac{\phi(x_1)-\phi(y)}{|x_1-y|^{n+\beta}}dy+C_{n,\beta}PV\int_{\Sigma_\lambda}\frac{\phi(x_1)-\phi(y^\lambda)}{|x-y^\lambda|^{n+\beta}}dy\\ \nonumber
&=&C_{n,\beta}PV\int_{\Sigma_\lambda}\frac{\phi(x_1)-\phi(y)}{|x_1-y|^{n+\beta}}dy+C_{n,\beta}PV\int_{\Sigma_\lambda}\frac{\phi(x_1)+\phi(y)}{|x_1-y^\lambda|^{n+\beta}}dy\\ \nonumber
&=&C_{n,\beta}PV\int_{\Sigma_\lambda}\frac{\phi(x_1)-\phi(y)}{|x_1-y|^{n+\beta}}dy+C_{n,\beta}PV\int_{\Sigma_\lambda}\frac{\phi(x_1)+\phi(y)}{|x_1-y^\lambda|^{n+\beta}}dy\\ \nonumber
&\leq&C_{n,\beta}PV\int_{\Sigma_\lambda}\frac{\phi(x_1)-\phi(y)}{|x_1-y^\lambda|^{n+\beta}}dy+C_{n,\beta}PV\int_{\Sigma_\lambda}\frac{\phi(x_1)+\phi(y)}{|x_1-y^\lambda|^{n+\beta}}dy\\ \nonumber
&\leq&C_{n,\beta}\int_{\Sigma_\lambda}\frac{2\phi(x_1)}{|x_1-y^\lambda|^{n+\beta}}dy.
\end{eqnarray}
Fix $\lambda$, from the fact $x_1\in\Sigma_\lambda$ and $|x_1|$ sufficiently large,
\begin{eqnarray}\label{1.7}\nonumber
\int_{\sum_\lambda}
\frac{1}{|x_1-y^\lambda|^{n+\beta}}dy&\geq&\int_{\{x_1\geq0\}}
\frac{1}{|x_1-y^\lambda|^{n+\beta}}dy\\ \nonumber
&=&\frac{1}{2}\int_{R^n}\frac{1}{(|x_1|+|y^\lambda|)^{n+\beta}}dy\\
\nonumber
&=&\frac{1}{2}\int_{0}^{\infty}\int_{B_r^0}\frac{1}{(|x_1|+|r|)^{n+\beta}}d\sigma
dr\\ \nonumber
&=&\frac{1}{2}\int_0^{\infty}\frac{w_{n-1}r^{n-1}}{(|x_1|+|r|)^{n+\beta}}dr\\
\nonumber
&=&\frac{w_{n-1}}{2|x_1|^\beta}\int_0^{\infty}\frac{t^{n-1}}{(1+t)^{n+\beta}}dr,           r=t|x_2|\\
&\sim&\frac{C}{|x_1|^\beta}.
\end{eqnarray}
Then,
\begin{eqnarray}\label{1.8}
(-\triangle)^{\beta/2}\phi(x_1)\leq\frac{C}{|x_1|^\beta}\phi(x_1)<0.
\end{eqnarray}
Combining this with (\ref{1.3}), we can show
\begin{eqnarray}\label{1.9}
\varphi(x_1)<0,
\end{eqnarray}
and
\begin{eqnarray}\label{1.10}
\phi(x_1)\geq -CC_1|x_1|^\beta\varphi(x_1).
\end{eqnarray}
We know there exists $x_0$ such that
$$\varphi(x_0)=\min_\Omega\varphi(x)<0.$$
From a similar argument as in (\ref{1.6}), we can show
\begin{eqnarray}\label{1.11}
(-\triangle)^{\alpha/2}\varphi(x_0)\leq\frac{C}{|x_0|^\alpha}\varphi(x_0).
\end{eqnarray}
From (\ref{1.3}) and (\ref{1.10}), we can deduce 
\begin{eqnarray*}
	0&\leq&(-\triangle)^{\alpha/2}\varphi(x_0)+C_2(x_0)\varphi(x_0)+C_3(x_0)\phi(x_0)\\
	&\leq&\frac{C}{|x_0|^\alpha}\varphi(x_0)+C_2(x_0)\varphi(x_0)+C_3(x_0)\phi(x_1)\\
	&\leq&\frac{C}{|x_0|^\alpha}\varphi(x_0)+C_2(x_0)\varphi(x_0)-CC_3(x_0)C_1(x_1)|x_1|^\beta\varphi(x_1)\\
	&<&0.
\end{eqnarray*}
The last inequality follows from assumptions (\ref{1.4}). This is a contradiction. Then (\ref{1.12}) must be true for at least one of $x_0$ and $x_1$.
\subsection{Narrow Region Principle}
\textbf{Proof}. If (\ref{2017010901}) does not hold, because $\phi(x)$ is lower semi-continuous, there exists $x_1\in\bar{\Omega}$ such that
\begin{eqnarray*}
	\phi(x_1)=\min_{\bar{\Omega}}\phi(x)<0.
\end{eqnarray*}
By the definition of $(-\triangle)^{\beta/2}$, we have
\begin{eqnarray}\label{010904}\nonumber
(-\triangle)^{\beta/2}\phi(x_1)&=&C_{n,\beta}PV\int_{R^n}\frac{\phi(x_1)-\phi(y)}{|x_1-y|^{n+\beta}}dy\\ \nonumber
&=&C_{n,\beta}PV\int_{\Sigma_\lambda}\frac{\phi(x_1)-\phi(y)}{|x_1-y|^{n+\beta}}dy+C_{n,\beta}PV\int_{\tilde{\Sigma}_{\lambda}}\frac{\phi(x_1)-\phi(y)}{|x_1-y|^{n+\beta}}dy\\ \nonumber
&=&C_{n,\beta}PV\int_{\Sigma_\lambda}\frac{\phi(x_1)-\phi(y)}{|x_1-y|^{n+\beta}}dy+C_{n,\beta}PV\int_{\Sigma_\lambda}\frac{\phi(x_1)-\phi(y^\lambda)}{|x-y^\lambda|^{n+\beta}}dy\\ \nonumber
&=&C_{n,\beta}PV\int_{\Sigma_\lambda}\frac{\phi(x_1)-\phi(y)}{|x_1-y|^{n+\beta}}dy+C_{n,\beta}PV\int_{\Sigma_\lambda}\frac{\phi(x_1)+\phi(y)}{|x_1-y^\lambda|^{n+\beta}}dy\\ \nonumber
&=&C_{n,\beta}PV\int_{\Sigma_\lambda}\frac{\phi(x_1)-\phi(y)}{|x_1-y|^{n+\beta}}dy+C_{n,\beta}PV\int_{\Sigma_\lambda}\frac{\phi(x_1)+\phi(y)}{|x_1-y^\lambda|^{n+\beta}}dy\\ \nonumber
&\leq&C_{n,\beta}PV\int_{\Sigma_\lambda}\frac{\phi(x_1)-\phi(y)}{|x_1-y^\lambda|^{n+\beta}}dy+C_{n,\beta}PV\int_{\Sigma_\lambda}\frac{\phi(x_1)+\phi(y)}{|x_1-y^\lambda|^{n+\beta}}dy\\ \nonumber
&\leq&C_{n,\beta}\int_{\Sigma_\lambda}\frac{2\phi(x_1)}{|x_1-y^\lambda|^{n+\beta}}dy.
\end{eqnarray}
Let $D=B_{2\delta}(x_1)\cap\tilde{\Sigma}_\lambda$, then
\begin{eqnarray}\label{010902}\nonumber
&&\int_{\Sigma_\lambda}\frac{1}{|x_1-y^\lambda|^{n+\beta}}dy\\ \nonumber
&\geq&\int_{D}\frac{1)}{|x_1-y|^{n+\beta}}dy\\ \nonumber
&\geq&\frac{1}{10}\int_{B_{2\delta}(x_1)}\int_{D}\frac{1)}{|x_1-y|^{n+\beta}}dy\\
&\sim&\frac{C}{\delta^\beta}.
\end{eqnarray}
Thus,
\begin{eqnarray}
(-\triangle)^{\beta/2}\phi(x_1)\leq \frac{C\phi(x_1)}{\delta^\beta}<0.
\end{eqnarray}
Combining this with (\ref{1.5}), we have
\begin{eqnarray}
-C_1(x_1)\varphi(x_1)\leq \frac{C\phi(x_1)}{\delta^\beta}.
\end{eqnarray}
We know there exists a $x_2$ such that 
$$\varphi(x_2)=\min_{\bar{\Omega}}\varphi(x)<0.$$
Similar to (\ref{010902}), we can derive that
\begin{eqnarray*}
	(-\triangle)^{\alpha/2}\varphi(x_2)\leq\frac{C\varphi(x_2)}{\delta^\alpha}<0.
\end{eqnarray*}
By (\ref{1.5}), for $\delta$ sufficiently small, we have
\begin{eqnarray*}
	0&\leq&(-\triangle)^{\alpha/2}\varphi(x_2)+C_2(x_2)\varphi(x_2)+C_3(x_2)\phi(x_2)\\
	&\leq&\frac{C\varphi(x_2)}{\delta^\alpha}+C_2(x_2)\varphi(x_2)+C_3(x_2)\phi(x_1)\\
&\leq&\frac{C\varphi(x_2)}{\delta^\alpha}+C_2(x_2)\varphi(x_2)-C_3(x_2)C_1(x_1)C\delta^\beta\phi(x_1)\\
&\leq&\frac{C\varphi(x_2)}{\delta^\alpha}+C_2(x_2)\varphi(x_2)-C_3(x_2)C_1(x_1)\delta^\beta C\phi(x_2)\\
&=&\frac{C\varphi(x_2)}{\delta^\alpha}(1+\frac{C_2(x_2)\delta^\alpha}{C}-C_3(x_2)C_1(x_1)\delta^\alpha\delta^\beta)\\
&<&0,
\end{eqnarray*}
which is a contradiction.
To prove (\ref{2017010902}), we suppose there exists $\bar{x}\in \Omega$ such that 
$$\phi(\bar{x})=0.$$
Then,
\begin{eqnarray}\nonumber
(-\triangle)^{\beta/2}\phi(\bar{x})&=&C_{n,\beta}PV\int_{R^n}\frac{\phi(\bar{x})-\phi(y)}{|\bar{x}-y|^{n+\beta}}dy\\ \nonumber
&=&C_{n,\beta}PV\int_{\Sigma_\lambda}\frac{-\phi(y)}{|\bar{x}-y|^{n+\beta}}dy+C_{n,\beta}PV\int_{\tilde{\Sigma}_{\lambda}}\frac{-\phi(y)}{|\bar{x}-y|^{n+\beta}}dy\\ \nonumber
&=&C_{n,\beta}PV\int_{\Sigma_\lambda}\frac{-\phi(y)}{|\bar{x}-y|^{n+\beta}}dy+C_{n,\beta}PV\int_{\Sigma_\lambda}\frac{-\phi(y^\lambda)}{|x-y^\lambda|^{n+\beta}}dy\\ \nonumber
&=&C_{n,\beta}PV\int_{\Sigma_\lambda}\frac{-\phi(y)}{|\bar{x}-y|^{n+\beta}}dy+C_{n,\beta}PV\int_{\Sigma_\lambda}\frac{\phi(y)}{|\bar{x}-y^\lambda|^{n+\beta}}dy\\ 
&=&C_{n,\beta}PV\int_{\Sigma_\lambda}\big(\frac{1}{|\bar{x}-y^\lambda|^{n+\beta}}-\frac{1}{|\bar{x}-y|^{n+\beta}}\big)\phi(y)dy. \label{011701}
\end{eqnarray}
If $\phi(x)\neq 0$, (\ref{011701}) implies that
$$(-\triangle)^{\beta/2}\phi(\bar{x})<0.$$
Combining this with (\ref{1.5}), we can derive 
$$\varphi(\bar{x})<0,$$
which is a contradiction with (\ref{2017010901}). Therefore $\phi(x)$ is identically $0$ in $\Sigma_\lambda$. Since 
$$\phi(x^\lambda)=-\phi(x),\,\,x\in\Sigma_\lambda,$$
it shows that
$$\phi(x)=0,\,\,x\in R^n,$$
then
$$(-\triangle)^{\beta/2}\phi(x)=0,\,\,x\in R^n.$$
From (\ref{1.5}), we know
$$\varphi(x)\leq0,\,\,x\in \Sigma_\lambda.$$
We already proved 
$$\varphi(x)\geq0,\,\,x\in \Sigma_\lambda.$$
It must hold
$$\varphi(x)=0,\,\,x\in \Sigma_\lambda.$$
Combining this with the fact 
$$\varphi(x^\lambda)=-\varphi(x),\,\,x\in\Sigma_\lambda,$$
we have
$$\varphi(x)\equiv 0,\,\,x\in R^n.$$
From a similar argument, we can show if $\varphi(x)$ is $0$ at one point in $\Sigma_\lambda$, then $\phi(x)$ and $\varphi(x)$ are identically 0 in $R^n$.

\section{The Symmetry of Positive Solutions}
Without any decay conditions on $u$ and $v$, we are not able to carry the method of moving planes on $u$ and $v$
directly. To circumvent this difficulty, we make a Kelvin transform.
For any $x_0\in R^n$, let
\begin{eqnarray*}
&\overline{u}(x)=\frac{1}{|x-x_0|^{n-\alpha}}u(\frac{x-x_0}{|x-x_0|^2}+x_0),\\
&\overline{v}(x)=\frac{1}{|x-x_0|^{n-\beta}}v(\frac{x-x_0}{|x-x_0|^2}+x_0).
\end{eqnarray*}
Without loss of generality, let $x_0=0$, then
\begin{eqnarray*}
	&\overline{u}(x)=\frac{1}{|x|^{n-\alpha}}u(\frac{x}{|x|^2}),\\
	&\overline{v}(x)=\frac{1}{|x|^{n-\beta}}v(\frac{x}{|x|^2}).
\end{eqnarray*}
Thus,
\begin{eqnarray*}
(-\triangle)^{\alpha/2}\overline{u}(x)&=&\frac{1}{|x|^{n+\alpha}}(-\triangle)^{\alpha/2}u(\frac{x}{|x|^2}),\\
&=&\frac{1}{|x|^{n+\alpha}}v(\frac{x}{|x|^2})u^{p-1}(\frac{x}{|x|^2}),\\
&=&\frac{1}{|x|^{\alpha+\beta-(p-1)(n-\alpha)}}\overline{v}(x)\overline{u}^{p-1}(x).
\end{eqnarray*}
  In a similar way, we have
\begin{eqnarray*}
(-\triangle)^{\beta/2}\overline{v}(x)=\frac{1}{|x|^{\alpha+\beta-(p-1)(n-\alpha)}}\overline{u}^p(x).
\end{eqnarray*}
Let $\gamma=\alpha+\beta-(p-1)(n-\alpha)$, and $(\ref{011802})$ becomes
\begin{eqnarray}\label{011803}
\left\{
\begin{array}{ll}
 (-\triangle)^{\alpha/2}\overline{u}(x)=|x|^{-\gamma}\overline{v}(x)\overline{u}^{p-1}(x),\,\,&x\in R^n,\\
 (-\triangle)^{\beta/2}\overline{v}=|x|^{-\gamma}\overline{u}^p(x),\,\,&x\in R^n,\\
 \overline{u}\geq0, \overline{v}\geq 0,\,\,&x\in R^n.
 \end{array}
 \right.
\end{eqnarray}

We first give some basic notations before starting moving the planes. Then we start moving planes on system $(\ref{011802})$.

Let
\begin{eqnarray*}
T_\lambda=\{x\in R^n\mid x_1=\lambda,\lambda\in R\}
\end{eqnarray*}
be the moving plane,
\begin{eqnarray*}
\Sigma_\lambda=\{x\in R^n\mid x_1<\lambda\}
\end{eqnarray*}
be the region to the left of the plane, and
\begin{eqnarray*}
x^\lambda=(2\lambda-x_1,x_2,...,x_n)
\end{eqnarray*}
be the reflection of the point $x = (x_1, x_2, \cdot\cdot\cdot ,
x_n)$ about the plane $\mathrm{T}_\lambda$.

Assume that $(\overline{u},\overline{v})$ solves the
fractional system $(\ref{011802})$. To compare the values of
$\overline{u}(x)$ with $\overline{u}(x^\lambda)$ and
$\overline{v}(x)$ with $\overline{v}(x^\lambda)$, we denote
\begin{eqnarray*}
\left\{
\begin{array}{ll}
U_\lambda(x) = \overline{u}(x^\lambda)-\overline{u}(x),\\
V_\lambda(x) = \overline{v}(x^\lambda)-\overline{v}(x).\\
\end{array}
 \right.
\end{eqnarray*}
Then system $(\ref{011802})$ becomes

\begin{eqnarray}\label{011804}
\left\{
\begin{array}{ll}
(-\triangle)^{\alpha/2} U_{\lambda}(x)=\frac{1}{|x^\lambda|^{\gamma}}\overline{v}(x^\lambda)\overline{u}^{p-1}(x^\lambda)-\frac{1}{|x|^{\gamma}}\overline{v}(x)\overline{u}^{p-1}(x),\\
(-\triangle)^{\beta/2}
V_{\lambda}(x)=\frac{1}{|x^\lambda|^{\gamma}}\overline{u}^{p}(x^\lambda)-\frac{1}{|x|^{\gamma}}\overline{u}^{p}(x).
\end{array}
 \right.
\end{eqnarray}

\subsection{Proof of Theorem 1.1}
Now,we start moving planes.
\subsubsection{Subcritical Case $\frac{n}{n-\alpha}\leq p<\frac{n+\beta}{n-\alpha}$.}
$\mathbf{Step.1}$: We show that when $\lambda$ sufficiently
negative,
\begin{eqnarray}\label{011805}
 U_\lambda(x), V_\lambda(x)\geq0, \,\forall x\in\Sigma_{\lambda}\setminus\{0^\lambda\}.
\end{eqnarray}
We claim that for $\lambda$ sufficiently negative, there exists a constant $C$ such that 
$$U_\lambda(x), V_\lambda(x)\geq C>0,\,\,x\in B_\varepsilon(0^\lambda)\backslash\{0^\lambda\},$$
we will prove it in Appendix.
Hence, there must be a point $\bar{x}$ such that
$$U_\lambda(\bar{x})=\min_{x\in\Sigma_\lambda }U_\lambda(x)<0.$$
Moreover,
\begin{eqnarray*}
&&(-\triangle)^{\alpha/2}U_\lambda(\bar{x})\\
&=&C_{n,\alpha} PV\int_{R^n}\frac{U_\lambda(\bar{x})-U_\lambda(y)}{|\bar{x}-y|^{n+\alpha}}dy\\
&=&C_{n,\alpha}PV\int_{\Sigma_\lambda}\frac{U_\lambda(\bar{x})-U_\lambda(y)}{|\bar{x}-y|^{n+\alpha}}dy+C_{n,\alpha}PV\int_{\tilde{\Sigma}_\lambda}\frac{U_\lambda(\bar{x})-U_\lambda(y)}{|\bar{x}-y|^{n+\alpha}}dy\\
&=&C_{n,\alpha}PV\int_{\Sigma_\lambda}\frac{U_\lambda(\bar{x})-U_\lambda(y)}{|\bar{x}-y|^{n+\alpha}}dy+C_{n,\alpha}PV\int_{\Sigma_\lambda}\frac{U_\lambda(\bar{x})-U_\lambda(y^\lambda)}{|\bar{x}-y^\lambda|^{n+\alpha}}dy\\
&=&C_{n,\alpha}PV\int_{\Sigma_\lambda}\frac{U_\lambda(\bar{x})-U_\lambda(y)}{|\bar{x}-y|^{n+\alpha}}dy+C_{n,\alpha}PV\int_{\Sigma_\lambda}\frac{U_\lambda(\bar{x})+U_\lambda(y)}{|\bar{x}-y|^{n+\alpha}}dy\\
&\leq&C_{n,\alpha}PV\int_{\Sigma_\lambda}\frac{U_\lambda(\bar{x})-U_\lambda(y)}{|\bar{x}-y^\lambda|^{n+\alpha}}dy+C_{n,\alpha}PV\int_{\Sigma_\lambda}\frac{U_\lambda(\bar{x})+U_\lambda(y)}{|\bar{x}-y|^{n+\alpha}}dy\\
&=&C_{n,\alpha}PV\int_{\Sigma_\lambda}\frac{2U_\lambda(\bar{x})}{|\bar{x}-y^\lambda|^{n+\alpha}}dy.
\end{eqnarray*}
From a similar argument in (\ref{1.8}), we have 
\begin{equation}\label{011902}
(-\triangle)^{\alpha/2}U_\lambda(\bar{x})\leq\frac{CU_\lambda(\bar{x})}{|\bar{x}|^\alpha}<0.
\end{equation}
We claim that
\begin{equation}\label{011901}
V_\lambda(\bar{x})<0.
\end{equation}
Indeed, if not, $V_\lambda(\bar{x})\geq 0$. From (\ref{011804}), we have
\begin{eqnarray*}
&&(-\triangle)^{\alpha/2}U_\lambda(\bar{x})\\
&=&\frac{1}{|\bar{x}^\lambda|^{\gamma}}\overline{v}(\bar{x}^\lambda)\overline{u}^{p-1}(\bar{x}^\lambda)-\frac{1}{|\bar{x}|^{\gamma}}\overline{v}(\bar{x})\overline{u}^{p-1}(\bar{x})\\
&\geq&\frac{1}{|\bar{x}^\lambda|^{\gamma}}\overline{v}(\bar{x}^\lambda)\overline{u}^{p-1}(\bar{x}^\lambda)-\frac{1}{|\bar{x}|^{\gamma}}\overline{v}(\bar{x})\overline{u}^{p-1}(\bar{x}^\lambda)\\
&\geq&\frac{1}{|\bar{x}^\lambda|^{\gamma}}\overline{v}(\bar{x}^\lambda)\overline{u}^{p-1}(\bar{x}^\lambda)-\frac{1}{|\bar{x}|^{\gamma}}\overline{v}(\bar{x})\overline{u}^{p-1}(\bar{x}^\lambda)\\
&\geq&\frac{1}{|\bar{x}|^{\gamma}}\overline{u}^{p-1}(\bar{x}^\lambda)V_\lambda(\bar{x})\\
&\geq&0.
\end{eqnarray*}
This is a contradiction with (\ref{011902}). Then (\ref{011901}) holds. And from (\ref{011901}), we know there exists $\tilde{x}$ such that 
$$V_\lambda({\tilde{x}})=\min_{\Sigma_\lambda}V_\lambda(x)<0.$$
Similar to (\ref{011901}), we can obtain
$$U_\lambda(\tilde{x})<0.$$
Therefore,
\begin{eqnarray}\nonumber
	&&(-\triangle)^{\alpha/2}U_\lambda(\bar{x})\\ \nonumber
	&=&\frac{1}{|\bar{x}^\lambda|^{\gamma}}\overline{v}(\bar{x}^\lambda)\overline{u}^{p-1}(\bar{x}^\lambda)-\frac{1}{|\bar{x}|^{\gamma}}\overline{v}(\bar{x})\overline{u}^{p-1}(\bar{x})\\ \nonumber
&=&\frac{1}{|\bar{x}^\lambda|^{\gamma}}\overline{v}(\bar{x}^\lambda)\overline{u}^{p-1}(\bar{x}^\lambda)-\frac{1}{|\bar{x}|^{\gamma}}\overline{v}(\bar{x})\overline{u}^{p-1}(\bar{x}^\lambda)+\frac{1}{|\bar{x}|^{\gamma}}\overline{v}(\bar{x})\overline{u}^{p-1}(\bar{x}^\lambda)-\frac{1}{|\bar{x}|^{\gamma}}\overline{v}(\bar{x})\overline{u}^{p-1}(\bar{x})\\ \nonumber
&\geq&\frac{1}{|\bar{x}|^{\gamma}}\big[\overline{v}(\bar{x}^\lambda)\overline{u}^{p-1}(\bar{x}^\lambda)-\overline{v}(\bar{x})\overline{u}^{p-1}(\bar{x}^\lambda)\big]+\frac{1}{|\bar{x}|^{\gamma}}\big[\overline{v}(\bar{x})\overline{u}^{p-1}(\bar{x}^\lambda)-\overline{v}(\bar{x})\overline{u}^{p-1}(\bar{x})\big]\\ \nonumber
&=&\frac{1}{|\bar{x}|^{\gamma}}\overline{u}^{p-1}(\bar{x}^\lambda)V_\lambda(\bar{x})+(p-1)\frac{1}{|\bar{x}|^{\gamma}}\overline{v}(\bar{x})\xi^{p-2}U_\lambda(\bar{x}),\,\,\,\,\xi\in[\overline{u}(\bar{x}^\lambda),\overline{u}(\bar{x})]\\ \label{012701}
&\geq&\frac{1}{|\bar{x}|^{\gamma}}\overline{u}^{p-1}(\bar{x})V_\lambda(\bar{x})+(p-1)\frac{1}{|\bar{x}|^{\gamma}}\overline{v}(\bar{x})\bar{u}(\bar{x})^{p-2}U_\lambda(\bar{x}),
\end{eqnarray}
and
\begin{eqnarray*}
&&(-\triangle)^{\beta/2}V_\lambda(\tilde{x})\\
	&=&\frac{1}{|\tilde{x}^\lambda|^{\gamma}}\overline{u}^{p}(\tilde{x}^\lambda)-\frac{1}{|\tilde{x}|^{\gamma}}\overline{u}^{p}(\tilde{x})\\
	&=&\frac{1}{|\tilde{x}^\lambda|^{\gamma}}\overline{u}^{p}(\tilde{x}^\lambda)-\frac{1}{|\tilde{x}|^{\gamma}}\overline{u}^{p}(\tilde{x}^\lambda)+\frac{1}{|\tilde{x}|^{\gamma}}\overline{u}^{p}(\tilde{x}^\lambda)-\frac{1}{|\tilde{x}|^{\gamma}}\overline{u}^{p}(\tilde{x})\\
	&=&(\frac{1}{|\tilde{x}^\lambda|^{\gamma}}-\frac{1}{|\tilde{x}|^{\gamma}})\overline{u}^{p}(\tilde{x}^\lambda)+p\frac{1}{|\tilde{x}|^{\gamma}}\eta^{p-1}U_\lambda(\tilde{x}),\,\,\,\eta\in[\overline{u}(\tilde{x}^\lambda),\overline{u}(\tilde{x})]\\
	&\geq&p\frac{1}{|\tilde{x}|^{\gamma}}\overline{u}^{p-1}(\tilde{x})U_\lambda(\tilde{x}).
\end{eqnarray*}
Let
\begin{eqnarray*}
C_1(\tilde{x})&=&p\frac{1}{|\tilde{x}|^{\gamma}}\overline{u}^{p-1}(\tilde{x})\\
&\sim&\frac{1}{|\tilde{x}|^{\gamma}}\frac{1}{|\tilde{x}|^{(n-\alpha)(p-1)}}\\
&\sim&\frac{1}{|\tilde{x}|^{\alpha+\beta}},\,\,|\tilde{x}| \,\,\mbox{large \,\,enough},
\end{eqnarray*}
\begin{eqnarray*}
C_2(\bar{x})&=&(p-1)\frac{1}{|\bar{x}|^{\gamma}}\overline{v}(\bar{x})\bar{u}(\bar{x})^{p-2}\\
&\sim&\frac{1}{|\bar{x}|^{\gamma}}\frac{1}{|\bar{x}|^{n-\beta}}\frac{1}{|\bar{x}|^{(n-\alpha)(p-2)}}\\
&\sim&\frac{1}{|\bar{x}|^{2\alpha}},\,\,|\bar{x}| \,\,\mbox{large \,\,enough},
\end{eqnarray*}
and 
\begin{eqnarray*}
C_3(\bar{x})&=&\frac{1}{|\bar{x}|^{\gamma}}\overline{u}^{p-1}(\bar{x})\\
&\sim&\frac{1}{|\bar{x}|^{\gamma}}\frac{1}{|\bar{x}|^{(n-\alpha)(p-1)}}\\
&\sim&\frac{1}{|\bar{x}|^{\alpha+\beta}},\,\,|\bar{x}| \,\,\mbox{large \,\,enough}.
\end{eqnarray*}
Then by Theorem \ref{t1}, if $\lambda$ sufficiently negative, there must be one of $U_\lambda(x)$ and $V_\lambda(x)$ positive in $\Sigma_\lambda\backslash\{0^\lambda\}.$
 Without loss of generality, we assume 
 $$U_\lambda(x)\geq0,\,\,x\in\Sigma_\lambda\backslash\{0^\lambda\}.$$
And we claim that
$$V_\lambda(x)\geq0,\,\,x\in\Sigma_\lambda\backslash\{0^\lambda\}.$$
If not, there exists $\tilde{x}$ such that 
$$V_\lambda(\tilde{x})=\min_{\Sigma_\lambda}V_\lambda(x)<0.$$
Then from a similar argument, we can show
$$0>\frac{CV_\lambda(\tilde{x})}{|\tilde{x}|^\beta}\geq(-\triangle)^{\beta/2}V_\lambda(\tilde{x})\geq p\frac{1}{|\tilde{x}|^\gamma}\bar{x}^{p-1}(\tilde{x})U_\lambda(\tilde{x})>0.$$
This is a contradiction. This completes Step 1.

$\mathbf{Step.2}$: Step 1 provides a starting point, from which we can now move the plane $T_\lambda$ to the right as long as $(\ref{011805})$ holds to its limiting position.

Let$$\lambda_0=\{\lambda\leq 0\mid U_\mu\geq0, V_\mu\geq 0, \forall x\in\Sigma_\mu\backslash\{0^\mu\},\mu\leq\lambda\}.$$
By definition,
$$U_{\lambda_0}(x), V_{\lambda_0}(x)\geq 0,\,\,\forall x\in \Sigma_{\lambda_0}\backslash\{0^{\lambda_0}\}.$$

(i): If $\lambda_0=0$, we can move $T_\lambda$ from $+\infty$ to the left and show that
$$U_{\lambda_0}(x), V_{\lambda _0}(x)\leq 0,\,\,x\in\Sigma_{\lambda_0}\backslash\{0^{\lambda_0}\},$$
with $\lambda_0=0$. We obtain
$$U_0(x)=V_0(x)\equiv 0,\,\,x\in\Sigma_0.$$

For more general Kelvin transform, through a similar argument we can show that
$$\lambda_0=x^0_1,$$
and
$$U_{\lambda_0}(x), V_{\lambda_0}(x)\equiv 0,\,\,x\in\Sigma_{\lambda_0}.$$
Since the $x_1$ direction and $x^0$ can be chosen arbitrarily, we have actually shown that $\bar{u}$ and $\bar{v}$ is radially symmetric about any point in $R^n$. For any $x^1, x^2\in R^n$, let the midcenter be the point of Kelvin transform
$$x^0=\frac{x^1+x^2}{2},$$
and
$$y^1=\frac{x^1-x^0}{|x^1-x^0|^2}+x^0,\,\,y^2=\frac{x^2-x^0}{|x^2-x^0|^2}+x^0.$$
Then,
$$\bar{u}(y^1)=\bar{u}(y^2),\,\,\bar{v}(y^1)=\bar{y^2}.$$
Thus,
$$u(x^1)=u(x^2),\,\,v(x^1)=v(x^2).$$
Since $x^1, x^2$ is chosen arbitrarily, $u$ and $v$ must be constant. From (\ref{1.1}), we know
$$(-\triangle)^{\alpha/2}u=0,$$
and 
$$(|x|^{\beta-n}\ast u^p)u^{p-1}>0,$$
a contradiction. Hence, $(u,v)=(0,0)$.

(ii): If $\lambda_0<0$, there must be two cases.

Case i: $$U_{\lambda_0}(x)=V_{\lambda_0}(x)\equiv 0,\,\,\forall x\in\Sigma_{\lambda_0}\backslash\{0^{\lambda_0}\}.$$
Indeed, we suppose there exists $\bar{x}$ such that 
$$U_{\lambda_0}(\bar{x})=\min_{\Sigma_{\lambda_0}}U_{\lambda_0}(x)=0.$$
Then it must be  true that 
\begin{equation}\label{012601}
U_{\lambda_0}(x)\equiv 0,\,\,\forall x\in\Sigma_{\lambda_0}.
\end{equation}
If not,
\begin{eqnarray*}
	(-\triangle)^{\alpha/2}U_{\lambda_0}(\bar{x})=C_{n,\alpha}PV\int_{R^n}\frac{-U_{\lambda_0}(\bar{x})}{|\bar{x}-y|^{n+\alpha}}dy<0.
\end{eqnarray*}
On the other hand,
\begin{eqnarray*}
	&&(-\triangle)^{\alpha/2}U_{\lambda_0}(\bar{x})\\
	&=&\frac{1}{|\bar{x}^{\lambda_0}|^{\gamma}}\overline{v}(\bar{x}^{\lambda_0})\overline{u}^{p-1}(\bar{x}^{\lambda_0})-\frac{1}{|\bar{x}|^{\gamma}}\overline{v}(\bar{x})\overline{u}^{p-1}(\bar{x})\\
	&\geq&\frac{1}{|\bar{x}^{\lambda_0}|^{\gamma}}\overline{v}(\bar{x}^{\lambda_0})\overline{u}^{p-1}(\bar{x}^{\lambda_0})-\frac{1}{|\bar{x}|^{\gamma}}\overline{v}(\bar{x})\overline{u}^{p-1}(\bar{x}^{\lambda_0})\\
	&\geq&\frac{1}{|\bar{x}^{\lambda_0}|^{\gamma}}\overline{v}(\bar{x}^{\lambda_0})\overline{u}^{p-1}(\bar{x}^{\lambda_0})-\frac{1}{|\bar{x}|^{\gamma}}\overline{v}(\bar{x})\overline{u}^{p-1}(\bar{x}^{\lambda_0})\\
	&\geq&\frac{1}{|\bar{x}|^{\gamma}}\overline{u}^{p-1}(\bar{x}^{\lambda_0})V_{\lambda_0}(\bar{x})\\
	&\geq&0,
\end{eqnarray*}
which is a contradiction. This proves (\ref{012601}).

Since
$$U_{\lambda_0}(x)=-U_{\lambda_0}(x^{\lambda_0}),$$
we have
$$U_{\lambda_0}(x)\equiv 0,\,\,x\in R^n.$$
Then,
\begin{eqnarray*}
	0&=&(-\triangle)^{\alpha/2}U_{\lambda_0}(x)\\
	&=&\frac{1}{|x^{\lambda_0}|^{\gamma}}\overline{v}(x^{\lambda_0})\overline{u}^{p-1}(x^{\lambda_0})-\frac{1}{|x|^{\gamma}}\overline{v}(x)\overline{u}^{p-1}(x)\\
	&=&\frac{1}{|x^{\lambda_0}|^{\gamma}}\overline{v}(x^{\lambda_0})\overline{u}^{p-1}(x)-\frac{1}{|x|^{\gamma}}\overline{v}(x)\overline{u}^{p-1}(x),	
\end{eqnarray*}
we can derive that
$$\overline{v}(x^{\lambda_0})\leq\overline{v}(x),\,\,x\in\Sigma_{\lambda_0}.$$
Combing this with the fact
$$\overline{v}(x^{\lambda_0})\geq\overline{v}(x),\,\,x\in\Sigma_{\lambda_0}.$$
We can deduce
$$V_{\lambda_0}(x)\equiv 0,\,\,x\in R^n.$$
From a similar argument, we can also have if $V_{\lambda_0}(x)=0$ somewhere, then
$$U_{\lambda_0}(x)=V_{\lambda_0}(x)\equiv 0,\,\,x\in R^n.$$
Therefore, for all $x\in R^n$,
\begin{eqnarray*}
	0&=&(-\triangle)^{\alpha/2}U_{\lambda_0}(x)\\
		&=&\frac{1}{|x^{\lambda_0}|^{\gamma}}\overline{v}(x^{\lambda_0})\overline{u}^{p-1}(x^{\lambda_0})-\frac{1}{|x|^{\gamma}}\overline{v}(x)\overline{u}^{p-1}(x),
\end{eqnarray*}
and 
\begin{eqnarray*}
	0&=&(-\triangle)^{\beta/2}V_{\lambda_0}(x)\\
&=&\frac{1}{|x^{\lambda_0}|^{\gamma}}\overline{u}^{p}(x^{\lambda_0})-\frac{1}{|x|^{\gamma}}\overline{u}^{p}(x).
\end{eqnarray*}
We can easily deduce that
$$\bar{u}(x)=\bar{v}(x)\equiv 0,\,\,x\in R^n.$$
Then 
$$u(x)=v(x)\equiv0,\,\,x\in R^n.$$

Case ii:  $$U_{\lambda_0}(x), V_{\lambda_0}(x)> 0,\,\,\forall x\in\Sigma_{\lambda_0}\backslash\{0^{\lambda_0}\}.$$
We show that the plane $T_\lambda$ can be moved further right. To be more rigorous, there exists some $\varepsilon>0$, such that for any $\lambda\in(\lambda_0,\lambda_0+\varepsilon)$, we have
$$U_\lambda(x), V_\lambda(x)\geq 0,\,\,x\in\Sigma_\lambda\backslash\{0^\lambda\},$$
This is a contradiction with the definition of $\lambda_0$, so this case will not happen. 

Indeed, first we claim that for $\lambda_0<0$ and $\varepsilon$ sufficiently small,
$$U_{\lambda_0}(x), V_{\lambda_0}(x)\geq C>0,\,\,x\in B_\varepsilon(0^{\lambda_0})\backslash\{0^{\lambda_0}\},$$
this will be proved in Appendix. Then there exist $\delta>0$ small and a constant $C>0$ such that 
$$U_{\lambda_0}(x), V_{\lambda_0}(x)\geq C>0,\,\,x\in (\Sigma_{\lambda_0-\delta}\backslash\{0^{\lambda_0}\})\cap B_R(0).$$
Since $U_\lambda(x), V_\lambda(x)$ are continuous about $\lambda$, then
$$U_{\lambda}(x), V_{\lambda}(x)\geq 0,\,\,x\in (\Sigma_{\lambda_0-\delta}\backslash\{0^{\lambda_0}\})\cap B_R(0).$$
Suppose $U_{\lambda}(x)<0,\,\,x\in \Sigma_{\lambda}\backslash\{0^{\lambda}\},$
then there must exist $\bar{x}$ such that
$$U_{\lambda}(\bar{x})=\min_{\Sigma_\lambda}U_\lambda(x)<0.$$
From a similar argument in the proof of \emph{Decay at Infinity}, we have
$$V_{\lambda}(\bar{x})<0,\,\,x\in\Sigma_{\lambda}\backslash\{0^{\lambda}\}.$$
Then there exists $\tilde{x}$ such that 
$$V_\lambda(\tilde{x})=\min_{\Sigma_\lambda}V_\lambda(x)<0.$$
By Theorem \ref{t1} (\emph{Decay at Infinity}), one of $\bar{x}$ and $\tilde{x}$ must
be in $B_R(0)$. Without loss of generality, we assume
$$|\bar{x}|<R.$$
Hence,
$$\bar{x}\in(\Sigma_\lambda\backslash\Sigma_{\lambda_0-\delta})\cap B_R(0).$$

If $\tilde{x}\in(\Sigma_\lambda\backslash\Sigma_{\lambda_0-\delta})\cap B_R(0)$, then by (\ref{011902}), we have,
$$(-\triangle)^{\alpha/2}U_\lambda(\bar{x})\leq\frac{CU_\lambda(\bar{x})}{(\delta+\varepsilon)^\alpha}<0.$$
Similarly, $$(-\triangle)^{\beta/2}V_\lambda(\tilde{x})\leq\frac{CV_\lambda(\tilde{x})}{(\delta+\varepsilon)^\beta}<0.$$
Then,
\begin{eqnarray*}
	0&\leq& (-\triangle)^{\alpha/2}V_\lambda(\bar{x})+C_2(\bar{x})U_\lambda(\bar{x})+C_3(\bar{x})V_\lambda(\bar{x})\\
	&\leq&\frac{CU_\lambda(\bar{x})}{(\delta+\varepsilon)^\alpha}+C_2(\bar{x})U_\lambda(\bar{x})+C_3(\bar{x})V_\lambda(\tilde{x})\\
	&\leq&\frac{CU_\lambda(\bar{x})}{(\delta+\varepsilon)^\alpha}+C_2(\bar{x})U_\lambda(\bar{x})+C_3(\bar{x})V_\lambda(\tilde{x})\\
	&\leq&\frac{CU_\lambda(\bar{x})}{(\delta+\varepsilon)^\alpha}+C_2(\bar{x})U_\lambda(\bar{x})-C_3(\bar{x})C_1(\tilde{x})(\delta+\varepsilon)^\beta U_\lambda(\tilde{x})\\
		&\leq&\frac{CU_\lambda(\bar{x})}{(\delta+\varepsilon)^\alpha}+C_2(\bar{x})U_\lambda(\bar{x})-C_3(\bar{x})C_1(\tilde{x})(\delta+\varepsilon)^\beta U_\lambda(\bar{x})\\
		&=&\frac{CU_\lambda(\bar{x})}{(\delta+\varepsilon)^\alpha}\{1-C_3(\bar{x})C_1(\tilde{x})(\delta+\varepsilon)^{\alpha+\beta}\}+C_2(\bar{x})U_\lambda(\bar{x})
\end{eqnarray*}
Through an identical argument in Theorem \ref{t2} (\emph{Narrow Region Principle}), we have
$$U_\lambda(x), V_\lambda(x)\geq 0,\,\,x\in (\Sigma_\lambda\backslash\Sigma_{\lambda_0-\delta})\cap B_R(0).$$
It implies that $\tilde{x}\in(\Sigma_\lambda\backslash\Sigma_{\lambda_0-\delta})\cap B_R(0)$ will not be happen.

If $\tilde{x}\in B_R^c\cap\Sigma_\lambda$, then
\begin{eqnarray*}
0&>&\frac{CU_\lambda(\bar{x})}{(\delta+\varepsilon)^\alpha}\\
&\geq&(-\triangle)^{\alpha/2}U_\lambda(\bar{x})\\
&\geq&\frac{1}{|\bar{x}|^{\gamma}}\overline{u}^{p-1}(\bar{x})V_\lambda(\bar{x})+(p-1)\frac{1}{|\bar{x}|^{\gamma}}\overline{v}(\bar{x})\bar{u}^{p-2}(\bar{x})U_\lambda(\bar{x})\\
&\geq&\frac{1}{|\bar{x}|^{\gamma}}\overline{u}^{p-1}(\bar{x})V_\lambda(\tilde{x})+(p-1)\frac{1}{|\bar{x}|^{\gamma}}\overline{v}(\bar{x})\bar{u}^{p-2}(\bar{x})U_\lambda(\bar{x}),
\end{eqnarray*}
and,
\begin{eqnarray*}
	0&>&\frac{CV_\lambda(\tilde{x})}{|\tilde{x}|^\beta}\\
	&\geq&(-\triangle)^{\beta/2}U_\lambda(\tilde{x})\\
	&\geq&p\frac{1}{|\tilde{x}|^{\gamma}}\overline{u}^{p-1}(\tilde{x})U_\lambda(\tilde{x})\\
	&\geq&p\frac{1}{|\tilde{x}|^{\gamma}}\overline{u}^{p-1}(\tilde{x})U_\lambda(\bar{x}).
\end{eqnarray*}
Through a simple calculation, we have
$$V_\lambda(\tilde{x})\geq|\tilde{x}|^{\beta-\gamma}\bar{u}^{p-1}(\tilde{x})U_\lambda(\bar{x}),$$
and,
\begin{eqnarray*}
U_\lambda(\bar{x})&\geq&(\delta+\varepsilon)^\alpha\big\{\frac{1}{|\bar{x}|^{\gamma}}\overline{u}^{p-1}(\bar{x})V_\lambda(\tilde{x})+(p-1)\frac{1}{|\bar{x}|^{\gamma}}\overline{v}(\bar{x})\bar{u}^{p-2}(\bar{x})U_\lambda(\bar{x})\big\}\\
&\geq&(\delta+\varepsilon)^\alpha\big\{\frac{1}{|\bar{x}|^{\gamma}}\overline{u}^{p-1}(\bar{x})|\tilde{x}|^{\beta-\gamma}\bar{u}^{p-1}(\tilde{x})U_\lambda(\bar{x})+(p-1)\frac{1}{|\bar{x}|^{\gamma}}\overline{v}(\bar{x})\bar{u}^{p-2}(\bar{x})U_\lambda(\bar{x})\big\}.
\end{eqnarray*}
Then,
\begin{eqnarray}\label{012705}\nonumber
1&\leq&(\delta+\varepsilon)^\alpha\big\{\frac{1}{|\bar{x}|^{\gamma}}\overline{u}^{p-1}(\bar{x})|\tilde{x}|^{\beta-\gamma}\bar{u}^{p-1}(\tilde{x})+(p-1)\frac{1}{|\bar{x}|^{\gamma}}\overline{v}(\bar{x})\bar{u}^{p-2}(\bar{x})\big\}\\
&=&(\delta+\varepsilon)^\alpha\big\{\frac{1}{|\bar{x}|^{\alpha+\beta}}u^{p-1}(\frac{\bar{x}}{|\bar{x}|^2})\frac{1}{|\tilde{x}|^\alpha}u^{p-1}(\frac{\tilde{x}}{|\tilde{x}|^2})+(p-1)\frac{1}{|\bar{x}|^{2\alpha}}v(\frac{\bar{x}}{|\bar{x}|^2})u^{p-2}(\frac{\bar{x}}{|\bar{x}|^2})\big\}.
\end{eqnarray}
For a fixed $\lambda_0<0$, when $\varepsilon$ is sufficiently small, we have $\lambda<\lambda+\varepsilon<\frac{\lambda_0}{2}$. Since $\lambda_0\in\Sigma_\lambda$, it deduces that $|\bar{x}|>-\frac{\lambda_0}{2}$. Notice that $|\tilde{x}|>R$, then $\frac{1}{|\bar{x}|^{\alpha+\beta}}u^{p-1}(\frac{\bar{x}}{|\bar{x}|^2})\frac{1}{|\tilde{x}|^\alpha}u^{p-1}(\frac{\tilde{x}}{|\tilde{x}|^2})+(p-1)\frac{1}{|\bar{x}|^{2\alpha}}v(\frac{\bar{x}}{|\bar{x}|^2})u^{p-2}(\frac{\bar{x}}{|\bar{x}|^2})$ is bounded. This shows that (\ref{012705}) must not be true for $\delta$ sufficiently small. This implies this case also will not happen.

\subsubsection{Critical Case $p=\frac{n+\beta}{n-\alpha}$.}
$\mathbf{Step.1}$: We show that when $\lambda$ sufficiently
negative,
\begin{eqnarray}\label{012801}
U_\lambda(x), V_\lambda(x)\geq0, \,\forall x\in\Sigma_{\lambda}\setminus\{0^\lambda\}.
\end{eqnarray}
We claim that for $\lambda$ sufficiently negative, there exists a constant $C$ such that 
$$U_\lambda(x), V_\lambda(x)\geq C>0,\,\,x\in B_\varepsilon(0^\lambda)\backslash\{0^\lambda\},$$
we will prove it in Appendix.
Hence, there must be a point $\bar{x}$ such that
$$U_\lambda(\bar{x})=\min_{x\in\Sigma_\lambda }U_\lambda(x)<0.$$
Moreover,
\begin{eqnarray*}
	&&(-\triangle)^{\alpha/2}U_\lambda(\bar{x})\\
	&=&C_{n,\alpha} PV\int_{R^n}\frac{U_\lambda(\bar{x})-U_\lambda(y)}{|\bar{x}-y|^{n+\alpha}}dy\\
	&=&C_{n,\alpha}PV\int_{\Sigma_\lambda}\frac{U_\lambda(\bar{x})-U_\lambda(y)}{|\bar{x}-y|^{n+\alpha}}dy+C_{n,\alpha}PV\int_{\tilde{\Sigma}_\lambda}\frac{U_\lambda(\bar{x})-U_\lambda(y)}{|\bar{x}-y|^{n+\alpha}}dy\\
	&=&C_{n,\alpha}PV\int_{\Sigma_\lambda}\frac{U_\lambda(\bar{x})-U_\lambda(y)}{|\bar{x}-y|^{n+\alpha}}dy+C_{n,\alpha}PV\int_{\Sigma_\lambda}\frac{U_\lambda(\bar{x})-U_\lambda(y^\lambda)}{|\bar{x}-y^\lambda|^{n+\alpha}}dy\\
	&=&C_{n,\alpha}PV\int_{\Sigma_\lambda}\frac{U_\lambda(\bar{x})-U_\lambda(y)}{|\bar{x}-y|^{n+\alpha}}dy+C_{n,\alpha}PV\int_{\Sigma_\lambda}\frac{U_\lambda(\bar{x})+U_\lambda(y)}{|\bar{x}-y|^{n+\alpha}}dy\\
	&\leq&C_{n,\alpha}PV\int_{\Sigma_\lambda}\frac{U_\lambda(\bar{x})-U_\lambda(y)}{|\bar{x}-y^\lambda|^{n+\alpha}}dy+C_{n,\alpha}PV\int_{\Sigma_\lambda}\frac{U_\lambda(\bar{x})+U_\lambda(y)}{|\bar{x}-y|^{n+\alpha}}dy\\
	&=&C_{n,\alpha}PV\int_{\Sigma_\lambda}\frac{2U_\lambda(\bar{x})}{|\bar{x}-y^\lambda|^{n+\alpha}}dy.
\end{eqnarray*}
From a similar argument in (\ref{1.8}), we have 
\begin{equation}\label{012802}
(-\triangle)^{\alpha/2}U_\lambda(\bar{x})\leq\frac{CU_\lambda(\bar{x})}{|\bar{x}|^\alpha}<0.
\end{equation}
We claim that
\begin{equation}\label{012803}
V_\lambda(\bar{x})<0.
\end{equation}
Indeed, if not, $V_\lambda(\bar{x})\geq 0$. From (\ref{011804}), we have
\begin{eqnarray*}
	&&(-\triangle)^{\alpha/2}U_\lambda(\bar{x})\\
	&=&\overline{v}(\bar{x}^\lambda)\overline{u}^{p-1}(\bar{x}^\lambda)-\overline{v}(\bar{x})\overline{u}^{p-1}(\bar{x})\\
	&\geq&\overline{v}(\bar{x}^\lambda)\overline{u}^{p-1}(\bar{x}^\lambda)-\overline{v}(\bar{x})\overline{u}^{p-1}(\bar{x}^\lambda)\\
	&\geq&\overline{u}^{p-1}(\bar{x}^\lambda)V_\lambda(\bar{x})\\
	&\geq&0.
\end{eqnarray*}
This is a contradiction with (\ref{012802}). Then (\ref{012803}) holds. And from (\ref{012803}), we know there exists $\tilde{x}$ such that 
$$V_\lambda({\tilde{x}})=\min_{\Sigma_\lambda}V_\lambda(x)<0.$$
Similar to (\ref{012803}), we can obtain
$$U_\lambda(\tilde{x})<0.$$
Therefore,
\begin{eqnarray}\nonumber
&&(-\triangle)^{\alpha/2}U_\lambda(\bar{x})\\ \nonumber
&=&\overline{v}(\bar{x}^\lambda)\overline{u}^{p-1}(\bar{x}^\lambda)-\overline{v}(\bar{x})\overline{u}^{p-1}(\bar{x})\\ \nonumber
&=&\overline{v}(\bar{x}^\lambda)\overline{u}^{p-1}(\bar{x}^\lambda)-\overline{v}(\bar{x})\overline{u}^{p-1}(\bar{x}^\lambda)+\overline{v}(\bar{x})\overline{u}^{p-1}(\bar{x}^\lambda)-\overline{v}(\bar{x})\overline{u}^{p-1}(\bar{x})\\ \nonumber
&=&\overline{u}^{p-1}(\bar{x}^\lambda)V_\lambda(\bar{x})+(p-1)\overline{v}(\bar{x})\xi^{p-2}U_\lambda(\bar{x}),\,\,\,\,\xi\in[\overline{u}(\bar{x}^\lambda),\overline{u}(\bar{x})]\\ \label{012805}
&\geq&\overline{u}^{p-1}(\bar{x})V_\lambda(\bar{x})+(p-1)\overline{v}(\bar{x})\bar{u}(\bar{x})^{p-2}U_\lambda(\bar{x}),
\end{eqnarray}
and
\begin{eqnarray*}
	&&(-\triangle)^{\beta/2}V_\lambda(\tilde{x})\\
	&=&\overline{u}^{p}(\tilde{x}^\lambda)-\overline{u}^{p}(\tilde{x})\\
	&=&p\frac{1}{|\tilde{x}|^{\gamma}}\eta^{p-1}U_\lambda(\tilde{x}),\,\,\,\eta\in[\overline{u}(\tilde{x}^\lambda),\overline{u}(\tilde{x})]\\
	&\geq&p\overline{u}^{p-1}(\tilde{x})U_\lambda(\tilde{x}).
\end{eqnarray*}
Let
\begin{eqnarray*}
	C_1(\tilde{x})&=&p\overline{u}^{p-1}(\tilde{x})\\
	&\sim&\frac{1}{|\tilde{x}|^{(n-\alpha)(p-1)}}\\
	&\sim&\frac{1}{|\tilde{x}|^{\alpha+\beta}},\,\,|\tilde{x}| \,\,\mbox{large \,\,enough},
\end{eqnarray*}
\begin{eqnarray*}
	C_2(\bar{x})&=&(p-1)\overline{v}(\bar{x})\bar{u}(\bar{x})^{p-2}\\
	&\sim&\frac{1}{|\bar{x}|^{n-\beta}}\frac{1}{|\bar{x}|^{(n-\alpha)(p-2)}}\\
	&\sim&\frac{1}{|\bar{x}|^{2\alpha}},\,\,|\bar{x}| \,\,\mbox{large \,\,enough},
\end{eqnarray*}
and 
\begin{eqnarray*}
	C_3(\bar{x})&=&\overline{u}^{p-1}(\bar{x})\\
	&\sim&\frac{1}{|\bar{x}|^{(n-\alpha)(p-1)}}\\
	&\sim&\frac{1}{|\bar{x}|^{\alpha+\beta}},\,\,|\bar{x}| \,\,\mbox{large \,\,enough}.
\end{eqnarray*}
Then by Theorem \ref{t1}, if $\lambda$ sufficiently negative, there must be one of $U_\lambda(x)$ and $V_\lambda(x)$ positive in $\Sigma_\lambda\backslash\{0^\lambda\}.$
Without loss of generality, we assume 
$$U_\lambda(x)\geq0,\,\,x\in\Sigma_\lambda\backslash\{0^\lambda\}.$$
And we claim that
$$V_\lambda(x)\geq0,\,\,x\in\Sigma_\lambda\backslash\{0^\lambda\}.$$
If not, there exists $\tilde{x}$ such that 
$$V_\lambda(\tilde{x})=\min_{\Sigma_\lambda}V_\lambda(x)<0.$$
Then from a similar argument, we can show
$$0>\frac{CV_\lambda(\tilde{x})}{|\tilde{x}|^\beta}\geq(-\triangle)^{\beta/2}V_\lambda(\tilde{x})\geq p\bar{x}^{p-1}(\tilde{x})U_\lambda(\tilde{x})>0.$$
This is a contradiction. This completes Step 1.

$\mathbf{Step.2}$: Step 1 provides a starting point, from which we can now move the plane $T_\lambda$ to the right as long as $(\ref{012801})$ holds to its limiting position.

Let$$\lambda_0=\{\lambda\leq0\mid U_\mu\geq0, V_\mu\geq 0, \forall x\in\Sigma_\mu\backslash\{0^\mu\},\mu\leq\lambda\}.$$
By definition,
$$U_{\lambda_0}(x), V_{\lambda_0}(x)\geq 0,\,\,\forall x\in \Sigma_{\lambda_0}\backslash\{0^{\lambda_0}\}.$$

$\mathbf{Case. i}$: $\lambda_0<0$. Similar to the subcritical case, one can show that 
$$U_{\lambda_0}(x), V_{\lambda_0}(x)\geq 0, \,x\in \Sigma_{\lambda_0}\backslash\{0^{\lambda_0}\}.$$
It follows that $x^0$ is not a singular point of $\bar{u}$ and $\bar{v}$ and hence 
$$u(x)=O(\frac{1}{|x|^{n-\alpha}}), v(x)=O(\frac{1}{|x|^{n-\beta}}),\,\,|x|\rightarrow\infty$$
This enables us to apply the method of moving plane to $u$ and $v$ directly and show that $u$ and $v$ are symmetric about some point in $R^n$.

$\mathbf{Case. ii}$: $\lambda_0=0$. Then by moving the planes from $+\infty$, we
derive that $\bar{u}$ and $\bar{v}$ are symmetric about the origin, and so do $u$ and $v$. In any case, $u$ and $v$ are symmetric about some point in $R^n$. 

This completes the proof.

\section{Appendices}
\begin{lemma}\label{lem5.1}
For $\lambda$ negative large, there exists a constant $C>0$, such that
\begin{equation}\label{05.1}
U_\lambda(x), V_\lambda(x)\geq C>0,\, x\in B_\varepsilon(0^\lambda)\backslash\{0^\lambda\}.
\end{equation}
\end{lemma}
\begin{proof}
For $x\in\Sigma_\lambda$, as $|x|\rightarrow -\infty$, it is easy to see that
\begin{equation}\label{05.2}
\overline{u}(x)\rightarrow 0.
\end{equation}
To prove (\ref{05.1}), it is sufficient to show
$$\overline{u}_\lambda(x)\geq C>0, \,x\in B_\varepsilon(0^\lambda)\backslash\{0^\lambda\}.$$
Or equivalently,
$$\overline{u}(x)\geq C>0, \,x\in B_\varepsilon(0)\backslash\{0\}.$$
Let $\eta$ be a smooth cut-off function such that $\eta\in[0,1]$ in $R^n$, $\mbox{supp}~ \eta\subset B_2$ and $\eta\equiv 1$ in $B_1$. Let
$$(-\triangle)^{\alpha/2}\phi(x)=\eta(x)v(x)u^{p-1}(x).$$
Then,
$$\phi(x)=C_{n,-\alpha}\int_{R^n}\frac{\eta(y)v(y)u^{p-1}(y)}{|x-y|^{n-\alpha}}dy
=C_{n,-\alpha}\int_{B_2(0)}\frac{\eta(y)v(y)u^{p-1}(y)}{|x-y|^{n-\alpha}}dy.$$
It is trivial for $|x|$ sufficiently large,
\begin{equation}\label{05.3}
  \phi(x)\sim\frac{1}{|x|^{n-\alpha}}.
\end{equation}
Since
\begin{eqnarray}\label{05.4}\left\{
\begin{array}{lll}
 (-\triangle)^{\alpha/2}(u-\phi)\geq 0,~~~&x\in B_R,\\
 (u-\phi)(x)\geq 0,~~~ &x\in B_R^c,
 \end{array}
 \right.
\end{eqnarray}
by the maximum principle, we have
$$(u-\phi)(x)\geq 0,\,x\in B_R,$$
thus
$$(u-\phi)(x)\geq0,\,x\in R^n.$$
For $|x|$ sufficiently large, from (\ref{05.3}), one can see that for some constant $C>0$,
\begin{equation}\label{05.5}
 u(x)\geq\frac{C}{|x|^{n-\alpha}}.
\end{equation}
Hence for $|x|$ small
$$u(\frac{x}{|x|^2})\geq C|x|^{n-\alpha},$$
and
$$\overline{u}(x)=\frac{1}{|x|^{n-\alpha}}u(\frac{x}{|x|^2})\geq C.$$
Together with (\ref{05.2}), it yields that
\begin{equation}\label{05.6}
  U_\lambda(x)\geq\frac{C}{2}>0, \,x\in B_\varepsilon(0^\lambda)\setminus\{0^\lambda\}.
\end{equation}
Through an identical argument, one can show that (\ref{05.6}) holds for $V_\lambda(x)$ as well.
\end{proof}

\begin{lemma}\label{5.3}
	Let $(u, v)$ be a pair of nonnegative solutions of (\ref{011802}), and $\bar{u}$, $\bar{v}$ be 
the Kelvin transform of $u$ and $v$, then $\bar{u}$ and $\bar{v}$ also satisfy 
	\begin{eqnarray*}\left\{
		\begin{array}{lll}
		\bar{u}(x)=\int_{R^n}\frac{|y|^{-\gamma}\bar{v}(y)\bar{u}^{p-1}(y)}{|x-y|^{n-\alpha}}dy\\
	 \bar{v}(x)=\int_{R^n}\frac{|y|^{-\gamma}\bar{u}^p(y)}{|x-y|^{n-\beta}}dy
		\end{array}
		\right.
	\end{eqnarray*} 
and vice versa.
\end{lemma}
\begin{proof} It is easy to see $(\bar{u}, \bar{v})$ is a pair of nonnegative solutions to (\ref{011803}). From (\ref{011801}), we have 
	$$ \bar{v}(x)=\int_{R^n}\frac{|y|^{-\gamma}\bar{u}^p(y)}{|x-y|^{n-\beta}}dy.$$
Then, we only need to show 
$$	\bar{u}(x)=\int_{R^n}\frac{|y|^{-\gamma}\bar{v}(y)\bar{u}^{p-1}(y)}{|x-y|^{n-\alpha}}dy.$$	
We first show that
\begin{eqnarray}\label{02.1}
\bar{u}(x)=c_1+\int_{R^n}\frac{|y|^{-\gamma}\bar{v}(y)\bar{u}^{p-1}(y)}{|x-y|^{n-\alpha}}dy.
\end{eqnarray}
Let
\begin{eqnarray}\label{02.2}
&\bar{u}_R(x)=\int_{B_R(0)}G_R(x,y)|y|^{-\gamma}\bar{v}(y)\bar{u}^{p-1}(y)dy,
\end{eqnarray}
where $G_R(x,y)$ is the Green's function of fractional Laplacian on
$B_R(0)$.

It is easy to see that
\begin{eqnarray}\label{02.3}
\left\{
\begin{array}{ll}
(-\triangle)^{\alpha/2}\bar{u}_R(x)=|x|^{-\gamma}\bar{v}(x)\bar{u}^{p-1}(x),\,\,&\mbox{in}\,\,B_R(0)\backslash\{0\},\\
\bar{u}_R(x)=0,\,\,\,\,\,&\mbox{on}\,\,B_R^c(0).
\end{array}
\right.
\end{eqnarray}
Let $\varphi_R(x)=\bar{u}(x)-\bar{u}_R(x)$, from
(\ref{011803}) and (\ref{02.3}), we have
\begin{eqnarray*}
	\left\{
	\begin{array}{ll}
		(-\triangle)^{\alpha/2}\varphi_R(x)=0,\,\,&\mbox{in}\,\,B_R(0),\\
		\varphi_R(x)\geq 0,\,\,\,\,\,&\mbox{on}\,\,B_R^c(0).
	\end{array}
	\right.
\end{eqnarray*}
By the Maximum Principle, we derive
\begin{eqnarray}\label{02.4}
\varphi_R(x)\geq 0,\,\,\,x\in R^n.
\end{eqnarray}
Therefore, when $R\rightarrow \infty$,

\begin{eqnarray}\label{02.5}
\bar{u}_R(x)\rightarrow \tilde{u}(x)=\int_{R^n}\frac{|y|^{-\gamma}\bar{v}(y)\bar{u}^{p-1}(y)}{|x-y|^{n-\alpha}}dy,
\end{eqnarray}

Moreover,
\begin{eqnarray}\label{02.6}
(-\triangle)^{\alpha/2}\tilde{u}(x)=|x|^{-\gamma}\bar{v}(x)\bar{u}^{p-1}(x),\,\,\,&x\in R^n,
\end{eqnarray}
Now let $\Phi(x)=\bar{u}(x)-\tilde{u}(x)$.
From (\ref{1.1}) and (\ref{02.6}), we have
\begin{eqnarray*}
	\left\{
	\begin{array}{ll}
		(-\triangle)^{\alpha/2}\Phi(x)=0,\,\,\,&x\in R^n,\\
		\Phi(x)\geq 0,\,\,\,\,&x\in R^n.
	\end{array}
	\right.
\end{eqnarray*}
From Proposition 2 in \cite{ZCCY}, we have
$$\Phi(x)=c_1.$$
Thus we proved (\ref{02.1}).

Next, we will show that $c_1=0$. If $c_1>0$, then from (\ref{02.1}) and the fact $p\geq\frac{n}{n-\alpha}$, we have
$$
v(x)=\int_{R^n}\frac{|y|^{-\gamma}u^{p}(y)}{|x-y|^{n-\beta}}dy\geq \int_{R^n}\frac{c_1^p|y|^{-\gamma}}{|x-y|^{n-\beta}}dy=\infty.
$$
But it is impossible. Hence $c_1=0$. Therefore,
\begin{eqnarray*}
	\left\{
	\begin{array}{ll}
		&\bar{u}(x)=\int_{R^n}\frac{|y|^{-\gamma}\bar{v}(y)\bar{u}^{p-1}(y)}{|x-y|^{n-\alpha}}dy,\\
		&\bar{v}(x)=\int_{R^n}\frac{|y|^{-\gamma}\bar{u}^p(y)}{|x-y|^{n-\beta}}dy.
	\end{array}
	\right.
\end{eqnarray*}
We complete our proof.
\end{proof}
\begin{lemma}\label{lem5.2}
For $\lambda_0<0$, if either of $U_{\lambda_0}$, $V_{\lambda_0}$ is not identically 0, then there exist some constant $C$ and $\varepsilon>0$ small such that
$$U_{\lambda_0}(x), V_{\lambda_0}(x)\geq C>0, \,\,x\in B_\varepsilon(0^{\lambda_0})\setminus\{0^{\lambda_0}\}.$$
\end{lemma}
\begin{proof} From Lemma \ref{5.3}, we have the integral equation
\begin{eqnarray*}
  V_{\lambda_0}(x)&=& \overline{v}_{\lambda_0}(x)-\overline{v}(x) \\
   &=& C_{n,\beta}\int_{\Sigma_{\lambda_0}}(|y^{\lambda_0}|^{-\gamma}\bar{u}^{p}(y^{\lambda_0})-|y|^{-\gamma}\bar{u}^{p}(y))
   (\frac{1}{|x-y|^{n+\beta}}-\frac{1}{|x-y^{\lambda_0}|^{n+\beta}})dy \\
   &\geq& C_{n,\beta}\int_{\Sigma_{\lambda_0}}\frac{\overline{u}_{\lambda_0}^{p}(y)-\overline{u}^{p}(y)}{|y|^\gamma}
   \cdot(\frac{1}{|x-y|^{n+\beta}}-\frac{1}{|x-y^{\lambda_0}|^{n+\beta}})dy.
\end{eqnarray*}
Since
$$U_{\lambda_0}(x)\not\equiv 0,\,x\in\Sigma_{\lambda_0},$$
there exists some $x_0$ such that
$U_{\lambda_0}(x_0)>0.$
Thus, for some $\delta>0$ small, it holds that
$$\overline{u}_{\lambda_0}^p(y)-\overline{u}^p(y)\geq C>0, \,y\in B_\delta(x_0).$$
Therefore,
\begin{equation}\label{5.7}
 V_{\lambda_0}(x)\geq\int_{B_\delta(x_0)}Cdy\geq C>0.
\end{equation}
In a same way, one can show that $U_{\lambda_0}(x)$ also satisfies (\ref{5.7}).
\end{proof}
\noindent{\bf Acknowledgement}

The research was supported by NSFC(NO.11571176). The authors would like to express sincere thanks
to the anonymous referee for his/her carefully reading the
manuscript and valuable comments and suggestions.

\bigskip

{\em Author's Addresses and Emails:}
\medskip

Pei Ma

Jiangsu Key Laboratory for NSLSCS

School of Mathematical Sciences

Nanjing Normal University

Nanjing, Jiangsu 210023, China;

Department of Mathematical Sciences

Yeshiva University

New York, NY, 10033, USA

mapei0620@126.com

\medskip

Jihui Zhang

Jiangsu Key Laboratory for NSLSCS

School of Mathematical Sciences

Nanjing Normal University

Nanjing, Jiangsu 210023, China

zhangjihui@njnu.edu.cn
\end{document}